\documentclass[12pt]{amsart}
	
\usepackage{amsmath,amssymb,bbm,wasysym,eurosym, url}

\let\phi\varphi
\let\epsilon\varepsilon
\let\theta\vartheta
\let\mathbb\mathbbm

\newcommand{\der}{\partial}
\newcommand{\cum}{{\textstyle \varint}}
\newcommand{\galg}{\mathcal{F}}
\newcommand{\bspc}{\mathcal{B}}
\newcommand{\bvp}[2]{\boxed{\begin{array}{l}#1\\#2\end{array}}}
\newcommand{\intdiffop}{\galg_\Phi[\der,\cum]}
\newcommand{\eqintdiffop}{\galg[\der,\cum_{\!\Phi}]}
\newcommand{\diffop}{\galg[\der]}
\newcommand{\cumop}{\galg[\cum]}
\newcommand{\eqcumop}{\galg[\cum_{\!\Phi}]}
\newcommand{\eqcumdiffop}{\galg[\cum_{\!\Phi} \der]}
\newcommand{\dirs}{\dotplus}
\newcommand{\orth}[1]{#1^\perp}
\newcommand{\id}{\operatorname{id}}
\newcommand{\act}{\cdot}
\newcommand{\RR}{\mathbb{R}}
\newcommand{\evl}{\text{\officialeuro}}
\newcommand{\inner}[2]{\langle #1 | #2 \rangle}
\newcommand{\sgn}{\operatorname{sgn}}
\newcommand{\fri}[1]{#1^\Diamond}

\newtheorem{theorem}{Theorem}
\newtheorem{lemma}[theorem]{Lemma}

\theoremstyle{definition}
\newtheorem{definition}[theorem]{Definition}
\newtheorem{remark}[theorem]{Remark}
 
\begin{document}


\title{Green's Functions for\\ Stieltjes Boundary Problems}

\author{M. Rosenkranz \and N. Serwa}

\date{16 April 2015}

\maketitle
\begin{abstract}
  Stieltjes boundary problems generalize the customary class of
  well-posed two-point boundary value problems in three independent
  directions, regarding the specification of the boundary conditions:
  (1) They allow more than two evaluation points. (2) They allow
  derivatives of arbitrary order. (3) Global terms in the form of
  definite integrals are allowed. Assuming the Stieltjes boundary
  problem is regular (a unique solution exists for every forcing
  function), there are symbolic methods for computing the associated
  Green's operator.

  In  the classical case of well-posed two-point boundary value
  problems, it is known how to transform the Green's operator into the
  so-called Green's function, the representation usually preferred by
  physicists and engineers. In this paper we extend this
  transformation to the whole class of Stieltjes boundary problems. It
  turns out that the extension (1) leads to more case distinction, (2)
  implies ill-posed problems and hence distributional terms, (3) has
  apparently no effect on the structure of the Green's function.
\end{abstract}

\section{Introduction}
\label{sec:intro}
Boundary problems for linear ordinary differential equations (LODEs)
or linear partial differential equations (LPDEs) are certainly among
the most important model types in the engineering
sciences. Interestingly, their systematic treatment in Symbolic
Computation started rather recently~\cite{Rosenkranz2005}. For
handling the central problems of solving and factoring boundary
problems, a \emph{differential algebra setting} for LODEs is employed
in~\cite{RosenkranzRegensburger2008a,RosenkranzRegensburger2008} and
for LPDEs
in~\cite{RosenkranzRegensburgerTecBuchberger2009,RosenkranzPhisanbut2013}.
An overarching abstract framework based only on Linear Algebra is
developed in~\cite{RegensburgerRosenkranz2009}. For the classical
treatment of boundary problems in Analysis, we refer
to~\cite{Duffy2001,Fokas2008,Stakgold2011,Teschl2012}.

In this paper we restrict ourselves to LODEs, where the ``industrial
standard'' for solving boundary problems is their so-called
\emph{Green's function}. This is in stark contrast to the
operator-based methodology used in the above references. In fact,
given a fundamental system, the algorithm
of~\cite{Rosenkranz2005,RosenkranzRegensburger2008a} computes the
solution of a boundary problem in the form of its \emph{Green's
  operator}. In the classical setting of well-posed two-point boundary
value problems (see Section~\ref{sec:stieltjes-bp}), this algorithm
admits an optional extra step for extracting the corresponding Green's
function. Our goal here is to extend this postprocessing step to the
considerably larger class of Stieltjes boundary problems (see
Section~\ref{sec:stieltjes-bp}).

One way to understand the relationship between Green's operators and
functions is to view the latter as a certain canonical form. For
making this precise we equip the ring of integro-differential
operators with a slightly different set of reduction rules favoring
multiply initialized integrals, leading to the ring of \emph{equitable
  integro-differential operators}
(Section~\ref{sec:equitable-intdiffop}).

A \emph{simple example} will make this clear---in fact the simplest of
all honest boundary problems~\cite[\S3.2]{Rosenkranz2005}. Given a
forcing function~$f \in C^\infty[0,1]$, we want to find~$u \in
C^\infty[0,1]$ such that
\begin{equation*}
  \bvp{u'' = f,}{u(0) = u(1) = 0.}
\end{equation*}
The Green's operator~$G\colon C^\infty[0,1] \to C^\infty[0,1]$ of this
problem is defined by~$Gf = u$. Using the standard reduction system
of~\cite{RosenkranzRegensburger2008a}, we would distinguish one
integral like~$\cum f := \cum_0^x f(\xi) \, d\xi$ and then obtain the
Green's operator in the canonical form
\begin{equation}
  \label{eq:ex-green-std}
  G = x \cum - \cum x + x \lfloor 1 \rfloor \cum x - x \lfloor 1 \rfloor \cum,
\end{equation}
where~$\lfloor \alpha \rfloor$ denotes the evaluation functional~$f
\mapsto f(\alpha)$ for any~$\alpha \in \RR$, in analogy to the
multiplier notation of~\cite{Rosenkranz2005}. For extracting the
Green's function, however, it is more useful to use the alternative
canonical form
\begin{equation}
  \label{eq:ex-green-equit}
  G = x \cum_{\!0}\, x - x \cum_{\!1}\, x - \cum_{\!0}\, x + x \cum_{\!1}
\end{equation}
where~$\cum_{\!\alpha} f := \cum_\alpha^x f(\xi) \, d\xi$ now denotes
the integral initialized at the point~$\alpha \in \{0, 1\}$. In fact,
this is the form given in~\cite{Rosenkranz2005}, and we shall see in
Section~\ref{sec:equitable-intdiffop} that the setting of
biintegro-differential operators used there is essentially a special
case of the equitable operator ring employed in this paper. The point
of the canonical form~\eqref{eq:ex-green-equit} is that it allows us
to apply the defining relation~$Gf(x) = \cum_0^1 \, g(x,\xi) \, f(\xi)
\, d\xi$ of the Green's function directly to obtain the latter as
\begin{equation}
  \label{eq:ex-green-op}
  g(x,\xi) = 
  \begin{cases}
    (x-1)\xi & \text{for $0 \le \xi \le x \le 1$,}\\
    x(\xi-1) & \text{for $0 \le x \le \xi \le 1$.}
  \end{cases}
\end{equation}
Heuristically speaking, one moves the~$\cum_{\!0}$ terms to the upper
and the~$\cum_{\!1}$ terms to the lower branch, at the same time
translating the ``$x$'' after the integrals into~$\xi$. Note
incidentally that~$g(x,\xi) = g(\xi,x)$ in the above Green's
function~\eqref{eq:ex-green-op}. As is well known in
Analysis~\cite[\S7]{CoddingtonLevinson1955} \cite[\S5]{Teschl2012}, this is a consequence of
the self-adjoint nature of this boundary problem---a topic that we
would wish to investigate in the future for the more general class
of Stieltjes boundary Problems (Section~\ref{sec:conclusion}).

We will elaborate on the above principles to \emph{generalize} it in
three ``orthogonal directions'': (1) We allow more than two evaluation
points, leading to an increased number of case branches. (2) Using
derivatives of arbitrary order in the boundary conditions leads to
distributional terms. (3) Boundary conditions with integral terms
(so-called ``non-local problems'') are also included; they do not lead
to further complications.

\section{Stieltjes Boundary Problems}
\label{sec:stieltjes-bp}
For giving a precise definition of the class of admissible boundary
problems, we follow the setting
of~\cite{RosenkranzRegensburger2008a}. Hence let~$(\galg, \der, \cum)$
be a fixed ordinary integro-differential $K$-algebra (here ordinary
means~$\ker \der = K$). Later we shall specialize this to~$\galg =
C^\infty(\RR)$, the real- or complex-valued smooth function.  This is
theoretically convenient but of course needs to be replaced by a
suitable constructive subalgebra for actual computations.

The \emph{ring of integro-differential operators} over~$\galg$,
introduced in~\cite[\S3]{RosenkranzRegensburger2008a}, will be denoted
here by~$\intdiffop$ to emphasize its dependence of the chosen set of
characters~$\Phi$, and also to mark the contrast to the equitable
operator ring~$\eqintdiffop$ to be introduced in
Section~\ref{sec:equitable-intdiffop}, where the integral operators
are parametrized by~$\Phi$. In the case of~$\galg = C^\infty(\RR)$,
these characters will be evaluations at given points of~$\RR$ so
that we may take~$\Phi \subseteq \RR$.

We recall the \emph{standard decomposition}
\begin{equation}
  \label{eq:std-decomp}
  \intdiffop = \diffop \dirs \cumop \dirs (\Phi),
\end{equation}
where~$\diffop$ denotes the subalgebra of differential operators (the
$K$-subalgebra of~$\intdiffop$ generated by~$\galg$ and~$\der$),
$\cumop$ the nonunital subalgebra of integral operators (the nonunital
$K$-subalgebra of~$\intdiffop$ generated by~$\galg$ and~$\cum$),
and~$(\Phi)$ the two-sided ideal of~$\intdiffop$ generated by the
characters in~$\Phi$. The corresponding \emph{right} ideal~$|\Phi) =
\Phi \cdot \intdiffop$ is known as the ideal of \emph{Stieltjes
  conditions}, and one may check that~$(\Phi)$ is in fact the
left~$\galg$-module generated by the Stieltjes conditions.

From the viewpoint of applications, Stieltes conditions~$\beta \in
(\Phi)$ are easier to comprehend in terms of their $\intdiffop$-normal
form: They can be described uniquely as sums
\begin{equation}
  \label{eq:stieltjes-cond}
  \beta = \sum_{\phi \in \Phi} \sum_{i \ge 0} a_{\phi,i} \phi \der^i
  + \sum_{\phi \in \Phi} \phi \cum f_\phi
\end{equation}
with only finitely many~$a_{\phi,i} \in K$ and~$f_\phi \in \galg$
nonzero. The double sum in~\eqref{eq:stieltjes-cond} is called the
\emph{local part} of~$\beta$, the subsequent sum its~\emph{global
  part}. In the important~$C^\infty(\RR)$ case with distinguished
integral~$\cum = \cum_0^x$, this yields
\begin{equation*}
  \beta(u) = \sum_{\phi,i} a_{\phi,i} u^{(i)}(\phi) + \sum_\phi \cum_0^\phi
  f_\phi(\xi) \, u(\xi) \, d\xi,
\end{equation*}
for certain~$a_{\phi,i} \in \RR$ and~$f_\phi \in C^\infty(\RR)$.

An $n$-th order \emph{Stieltjes boundary problem} is a pair~$(T,
\bspc)$ with a monic differential operator~$T \in \diffop$ of order~$n$
and a boundary space~$\bspc \le \galg^*$ given as linear span~$\bspc = [\beta_1, \dots, \beta_n]$
of~$n$ linearly
independent Stieltjes conditions. In traditional representation, such a boundary problem is
displayed as
\begin{equation}
  \label{eq:bvp}
  \bvp{Tu=f,}{\beta_1 u = \cdots = \beta_n u = 0,}
\end{equation}
with the understanding that~$u \in \galg$ is desired for any
prescribed forcing function~$f \in \galg$. For the (usual) Green's
operator to be well-defined, we need the boundary
problem~\eqref{eq:bvp} to be \emph{regular} in the sense that~$\ker T
\dirs \orth{\bspc} = \galg$, where~$\orth{\bspc} = \{ u \in \galg \mid \beta(u) = 0 \;
\text{for all $\beta \in \bspc$} \}$ is the
corresponding space of admissible functions. Regularity is equivalent
to the requirement that~\eqref{eq:bvp} has a unique solution~$u \in
\galg$ for every given~$f \in \galg$. An algorithmic method for
testing regularity starts from a fundamental system~$u_1, \dots, u_n
\in \galg$ for~$T$, meaning a $K$-basis of~$\ker
T$. Then~\eqref{eq:bvp} is regular iff the evaluation matrix
\begin{equation}
  \label{eq:ev-matrix}
  \beta(u) = 
  \begin{pmatrix}
    \beta_1(u_1) & \cdots & \beta_1(u_n)\\
    \vdots & \ddots & \vdots\\
    \beta_n(u_1) & \cdots & \beta_n(u_n)
  \end{pmatrix}
  \in K^{n \times n}
\end{equation}
is regular; see~(15) of \cite{RosenkranzRegensburger2008a}. Given a fundamental system~$u_1, \dots, u_n$ for~$T$, the
solution algorithm of~\cite{RosenkranzRegensburger2008a} computes the
Green's operator of any regular Stieltjes boundary problem as an
integro-differential operator~$G \in \intdiffop$.

Within the class of Stieltjes boundary problems, we make the following
distinctions in order to characterize the \emph{classical scenario} as
a certain special case.

\begin{definition}
  \label{def:well-posed}
  A Stieltjes boundary problem~$(T, \bspc)$ of order~$n$ with~$\bspc =
  [\beta_1, \dots, \beta_n]$ is called \emph{well-posed} if
  the~$\beta_i$ can be chosen with all derivatives having order
  below~$n$; otherwise it is called \emph{ill-posed}. Furthermore, we
  call~$(T, \bspc)$ an \emph{$m$-point boundary problem} if the
  maximal number of evaluation points occurring in any
  $K$-basis~$(\beta_i)$ of~$\bspc$ is~$m$, and we call~$(T, \bspc)$
  \emph{local} if the $\beta_i$ can be chosen without global parts.
\end{definition}

Let us digress a bit on the notion of ill-posed boundary
problems. Following Hadamard, a problem is generally
called~\emph{well-posed}~\cite[p.~86]{Engl1997} if it is regular
(meaning its solution~$u$ exists and is unique for all given data~$f$)
as well as stable (meaning~$u$ depends continuously
on~$f$). Otherwise, one speaks of an \emph{ill-posed} problem. In the
case of boundary problems~\eqref{eq:bvp}, we search for~$u \in
C^\infty(\RR)$, and the data is given by the forcing function~$f \in
C^\infty(\RR)$. Stability---and hence well-posedness---depends on the
topology chosen for the function space~$C^\infty(\RR)$. Using the
$L^2$ norm as in many application problems, the distinction between
well- and ill-posed boundary problems coincides with the one given
above.

Since local boundary problems involve only evaluations of the unknown
function (rather than definite integrals), we also call them 
``boundary \emph{value} problems''. We can now characterize the
classical case, described for example
in~\cite[\S7]{CoddingtonLevinson1955}, by the following three-fold
restriction: They are the \underline{well-posed} \underline{two}-point
boundary \underline{value} problems. (Sometimes one meets the further
restriction to self-adjoint boundary problems.)

The classical case (in the above sense) is clearly the most frequent
case in the applications (but this could also be due to a selection
bias: having a well-equipped toolbox for classical problems might
tempt engineers to restrict their attention to classical
problems). Nevertheless, \emph{multi-point boundary value problems}
are also important for some applications~\cite[Ex.~1.6]{Agarwal1986},
\cite{SunRen2010}, \cite{Loud1968}, \cite{PangDongWei2006},
\cite{Beesack1962}. Boundary problems with nonlocal conditions are
more seldom, they are usually studied for nonlinear
equations~\cite{BenchohraHendersonLucaOuahab2014},
\cite{Jankowski2002}; the linear case serves as the initial
approximation. Finally, the case of ill-posed boundary problems
is---for obvious reasons---mostly avoided when engineering problems
are modelled. However, there are cases where their treatment is
inevitable, typically in the context of inverse
problems~\cite{EnglHankeNeubauer1996a}. Since the numerical treatment
of such problems is very delicate, it is of paramount importance to
have exact symbolic algorithms wherever this is possible.

We will lift all three of these restrictions for the algorithm of
\emph{extracting Green's functions}, which will be given below
(Section~\ref{sec:extract-gf}). As indicated in the Introduction, the
crucial tool for this purpose---even in the classical case---is the
ring of equitable integro-differential operators with its alternative
canonical forms.

\section{Equitable Operators}
\label{sec:equitable-intdiffop}
The passage from the standard integro-differential operator
ring~$\intdiffop$ to its equitable variant~$\eqintdiffop$ is based on
the \emph{fundamental theorem of calculus}~$\cum_{\!\phi}^x\, f'(\xi)
\, d\xi = f(x) - f(\phi)$ for any function~$f \in C^\infty(\RR)$ and
initialization point~$\phi \in \RR$. Likewise, if~$(\galg, \der,
\cum)$ is an arbitary integro-differential $K$-algebra and~$\phi$ a
character (multiplicative linear functional), one can use the
definition~$\cum_{\!\phi} := (\id-\phi) \cum$ to obtain the
corresponding relation~$\cum_{\!\phi} \, \der = \id - \phi$. In some
contexts (especially in the presence of several integral operators
like~$\cum^x$ and~$\cum^y$ on bivariate functions), it may be useful
to write the integral~$\cum_{\!\phi}$ as~$\cum_{\!\phi}^x$. If~$\psi$
is another character, one observes the relation~$\psi \cum_{\!\phi} =
\cum_{\!\psi} - \cum_{\!\phi}$, and it is natural to
write~$\cum_{\!\phi}^\psi$ for both expressions.

Note that~$(\galg, \der, \cum_{\!\phi})$ is again an ordinary
integro-differential $K$-algebra, and the preference of~$\cum$
over~$\cum_{\!\phi}$ can appear arbitrary in certain
settings. Accordingly, one may build the ring of integro-differential
operators by adjoining \emph{all} $\cum_{\!\phi}$ while the
characters~$\phi$ themselves are now redundant due to the above
fundamental relation. The precise formulation of the resulting
ring~$\eqintdiffop$ as a quotient is described
in~\cite[\S5.1]{RegensburgerRosenkranz2009a}. For our present
purposes, we shall only list its relations (see
Table~\ref{fig:equitable-rules} where~$\cdot$ denotes the natural
action of the operators), which are an easy consequence of the
relations of the standard integro-differential operator
ring~$\intdiffop$.

\def\qd{\kern0.5em}
\begin{table}[h]
  \newcommand{\ra}{\rightarrow}
  \centering
  \renewcommand{\baselinestretch}{1.5}
  \small
  \begin{tabular}[h]{|@{\qd}lcl@{\qd}|@{\qd}lcl@{\qd}|@{\qd}lcl@{\qd}|}
    \hline
    $f g$ & $\ra$ & $f \act g$ &
    $\der f$ & $\ra$ & $\der \act f + f \der$ &
    $\der\smash{\cum_{\!\phi}^x}$ & $\ra$ & $1$\\\hline
    $\smash{\cum_{\!\phi}^x} \, f \smash{\cum_{\!\psi}^x}$ & $\ra$ &
    \multicolumn{7}{l|}{$(\smash{\cum_{\!\phi}^x} \act f) 
      \smash{\cum_{\!\psi}^x} - \smash{\cum_{\!\phi}^x} \,
      (\smash{\cum_{\!\phi}^x} \act f)$}\\\hline
    $\smash{\cum_{\!\phi}^x} f \der$ & $\ra$ &
    \multicolumn{7}{l|}{$f - \smash{\cum_{\!\phi}^x} (\der \act f) -
      \phi \act f + (\phi \act f) \smash{\cum_{\!\phi}^x} \der$}\\
    \hline
  \end{tabular}
  \medskip
  \caption{Equitable Integro-Differential Relations}
  \label{fig:equitable-rules}
\end{table}

Similar to the standard decomposition~\eqref{eq:std-decomp}, we have
also the \emph{equitable decomposition} $\eqintdiffop = \diffop \dirs
\eqcumop \dirs \eqcumdiffop$ where~$\eqcumop$ is the nonunital
subalgebra of equitable integral operators~$\sum_{i=0}^n f_i
\cum_{\!\phi}^x g_i$ and~$\eqcumdiffop$ the $\galg$-submodule
consisting of~$\sum_{i=0}^n f_i \cum_{\!\phi}^x \der^i$; this leads to
the obvious normal forms in~$\eqintdiffop$.

The so-called \emph{translation isomorphism}~$\iota\colon \intdiffop \to
\eqintdiffop$ leaves $f \in \galg$ and~$\der$ invariant while using
the above fundamental relation via~$\iota(\phi) = \id - \cum_{\!\phi}
\der$ and~$\iota^{-1}(\cum_{\!\phi}) = (\id-\phi)\cum$. Note that this
holds also for the character~$\evl := \id - \cum \der$ associated to
the distinguished integral~$\cum = \cum_{\!\evl}$
underlying~$\intdiffop$. 

Specializing to~$\galg = C^\infty(\RR)$ and~$\Phi = \{0,1\}$, we can
deal with the \emph{example} in the Introduction, where we have the
normal form~$G \in \intdiffop$ in~\eqref{eq:ex-green-std} along with
its equitable variant~$\iota(G) \in \eqintdiffop$
in~\eqref{eq:ex-green-equit}. In such two-point cases with
characters~$\Phi = \{\alpha, \beta\}$, the equitable operator
ring~$\eqintdiffop$ is essentially the same as the ring of
\emph{biintegro-differential operators}. More precisely, we obtain a
biintegro-differential algebra~$(\galg, \der, \cum^*, \cum_*)$ with
integral~$\cum^* := \cum_\alpha$ and cointegral~$\cum_* :=
-\cum_\beta$ in the sense
of~\cite[Def.~3.23]{RegensburgerRosenkranz2009a}. Note that~$\cum^*$
and~$\cum_*$ are adjoint with respect to the inner product defined
by~$\inner{f}{g} := (\cum^*+\cum_*)(fg) = \cum_\alpha^\beta
fg$. Incidentally, the notion of biintegro-differential algebra
coincides with the (badly named) notion of ``analytic algebra''
introduced in~\cite[Def.~2]{Rosenkranz2005} and replicated
in~\cite[Ex.~5]{RosenkranzRegensburger2008a}. Clearly, the operator
ring resulting from~$\diffop$ by adjoing~$\cum^*$ and~$\cum_*$ is the
same as~$\eqintdiffop$, modulo the sign change in the cointegral.

Note that also $\cum_{\!\phi} \in \intdiffop$ and~$\phi \in
\eqintdiffop$ are legitimate expressions via the above translation
isomorphism. They are not in canonical form but we may think of them
as a kind of \emph{abbreviation} for the corresponding canonical
expression.

In fact, the extraction of Green's functions is based on the following
slight variation of the equitable integro-differential operator
ring~$\eqintdiffop$. Writing any element~$U \in \intdiffop$ in the
form~$U = T + K + B$ with~$T \in \diffop$, $K \in \cumop$ and~$B \in
(\Phi)$ according to~\eqref{eq:std-decomp}, we let~$T \in
\eqintdiffop$ and~$K \in \galg[\cum_{\!\evl}] \subseteq \eqcumop$
invariant while translating~$B \in (\Phi)$ as follows. Since~$(\Phi)$
is the left~$\galg$-module generated by Stieltjes
conditions~\eqref{eq:stieltjes-cond}, we may split~$B = \lambda +
\gamma$ into a left~$\galg$-linear combination~$\lambda$ of local
Stieltjes conditions and a left~$\galg$-linear combination~$\gamma$ of
global Stieltjes conditions. It turns out to be expedient to
keep~$\lambda$ in this form, without eliminating the characters
via~$\phi = \id - \cum_{\!\phi} \der$, but to translate~$\gamma$
via~$\phi \cum_{\!\evl} = \cum_{\!\evl} - \cum_{\!\phi} =:
\cum_{\!\evl}^\phi$. This is what we mean when referring in the sequel
to the \emph{equitable form} of an integro-differential operator~$U$.

\section{Extracting Green's Functions}
\label{sec:extract-gf}
We now turn to the central task of this paper, the extraction of the
Green's function~$g(x,\xi)$ corresponding to the Green's operator~$G
\in \eqintdiffop$ computed by the algorithm
of~\cite{RosenkranzRegensburger2008a} and converted to equitable form
as described in Section~\ref{sec:equitable-intdiffop}. Hence we
specialize now to~$\galg = C^\infty(\RR)$. Note that we may think
of~$g(x,\xi)$ as a kind of \emph{coordinate representation} of the
induced operator action~$G\colon \galg \to \galg$; in quantum
mechanics this would correspond to the ``position basis'' (as opposed
to the ``momentum basis'' in the Pontryagin dual reached via the
Fourier transform). Hence we will use the notation~$g(x,\xi) =
G_{x\xi}$, thinking of the~$x,\xi$ rather like continuous indices
similar to the discrete indices~$i,j$ in the matrix elements~$A_{ij}$
of some~$A \in K^{n \times n}$.

In fact, we will use this notation~$G_{x\xi}$ for any equitable
integro-differential operator~$G \in \eqintdiffop$. Its result will in
general contain Dirac distributions~\cite[\S2]{Stakgold2011} and their derivatives but nothing
beyond that. Since all boundary problems considered in this paper have
only finitely many evaluation points~$\alpha \in \Phi \subset \RR$,
one may choose an interval~$J \subset \RR$ containing all
the~$\alpha$. Hence the~$C(J^2)$-module~$\mathcal{G} \subset \mathcal{D}'(J^2)$ generated by
the \emph{Dirac distributions}~$\delta_\alpha$ and their derivatives will be sufficient to
capture all Green's ``functions'' $G_{x\xi} \in \mathcal{G}$. Here and
in the sequel we shall follow the common engineering (and also applied
maths) practice of referring to distributions like~$\delta_\alpha$ as
functions. In the same vein, we shall also write~$\delta(\xi-\alpha)$
in place of~$\delta_\alpha$, in view of the defining
property~$\cum_{\!J}\, \delta(\xi-\alpha) \, f(\xi) \, d\xi =
f(\alpha)$.

The transformation from Green's operators to Green's functions
\begin{equation*}
  \eqintdiffop \to \mathcal{G}, \; G \mapsto G_{x\xi}
\end{equation*}
is clearly an $\RR$-linear map, hence it will be sufficient to define it on the
\emph{canonical $\RR$-basis} of~$\intdiffop \cong
\eqintdiffop$. Following the strategy of the example in the
Introduction, the easiest part is~$\cumop \subseteq \intdiffop$, which
is handled by setting
\begin{equation*}
  (f \cum g)_{x\xi} = f(x) \, g(\xi) \, [0 \le \xi \le x]
  - f(x) \, g(\xi) \, [x \le \xi \le 0],
\end{equation*}
where we use the Iverson bracket notation~$[P]$ signifying~$1$ if the
property~$P$ is true and zero otherwise. Note that at most one of the
two summands above is nonzero for fixed~$(x,\xi)$. Since~$(\Phi)
\subset \intdiffop$ is a left~$\galg$-module over~$|\Phi)$, we settle
this part via
\begin{align*}
  (f \, \lfloor \alpha \rfloor \der^i)_{x\xi} &= (-1)^i f(x) \,
  \delta^{(i)}(\xi-\alpha),\\
  (f \, \lfloor \alpha \rfloor \cum g)_{x\xi} &= \sgn(\alpha) \, f(x) \,
  g(\xi) \, [0 \le \xi \le \alpha] .
\end{align*}
Finally, on~$\diffop$ we define
\begin{equation*}
  (f \der^i)_{x\xi} = (-1)^i \, f(x) \, \delta^{(i)}(x-\xi),
\end{equation*}
and the definition is complete in view
of~\eqref{eq:std-decomp}. Moreover, it is easy to check that the
assignment~$G \mapsto G_{x\xi}$ is correct in the sense that~$Gf =
\cum_{\!J}\, G_{x\xi} \, f(\xi) \, d\xi$. The isomorphism~$\iota$ of
Section~\ref{sec:equitable-intdiffop} may now be employed to obtain
the required transformation~$\eqintdiffop \to \mathcal{G}$. In fact, the
above case~$f \cum g \in \cumop$ generalizes immediately to
\begin{equation*}
  (f \cum_{\!\alpha}\, g)_{x\xi} = f(x) \, g(\xi) \, [\alpha \le \xi \le x]
  - f(x) \, g(\xi) \, [x \le \xi \le \alpha],
\end{equation*}
which will turn out to be the essential clause for extracting Green's
functions of (well-posed) multi-point boundary problems. For seeing
this, we need a more detailed description of the underlying Green's
operators.

We turn first to the easy case of a one-point boundary problem, more
appropriately known under the name of \emph{initial value
  problems}~$(T,[\evl, \dots, \evl \der^{n-1}])$ for~$T \in \diffop$
of order~$n$. The corresponding Green's operator is called the
fundamental right inverse~$\fri{T}$ and can be computed easily via the well-known ``variation of constants'' formula~\cite[Thm.~6.4]{RegensburgerRosenkranz2009a}: If~$u_1, \dots, u_n$ is a fundamental system for~$T$ with
  Wronskian matrix~$W$, the fundamental right inverse is given by
\begin{equation*}
  \fri{T} = \sum_{j=1}^n u_{j} \cum \frac{d_j}{d}.
\end{equation*}
Here~$d = \det(W)$ and~$d_i=\det(W_i)$, where~$W_i$ denotes the
  matrix resulting from~$W$ when replacing the~$i$-th column by the
  $n$-th unit vector of~$K^n$.

What we shall need in the sequel
is how~$\fri{T}$ reacts to left multiplication by~$\diffop$.
\begin{lemma}
  Let~$T \in \diffop$ be any monic differential operator of order~$n$,
  and choose a fundamental system~$u_1, \dots, u_n$ for~$T$ with
  Wronskian matrix~$W$. Then we have
  \begin{align}
    \label{eq:der-fri}
    \der^k &\fri{T} = \sum_{j=1}^n u_{j}^{(k)} \cum \frac{d_j}{d} + 
    \sum_{j=1}^k \der^{k-j} \rho_j\\
    & \text{with}\qquad
    \rho_k := \frac{1}{d} \sum_{j=1}^n u_j^{(k-1)} d_j \in \galg,\notag
  \end{align}
  where~$d$ and~$d_i$ are as above.
\end{lemma}

Note that~$\rho_1 = \cdots = \rho_{n-1} = 0$ by the definition of
the~$d_j$; hence the second sum in~$\der^k \fri{T}$ is only present
for~$k \ge n$, and we may equivalently write its range as~$j=n, \dots,
k$. Furthermore, we have~$\rho_n = 1$ from the definition
of~$d$. For~$k>n$, however, the~$\rho_k$ are functions of~$\galg$, so
in general they \emph{do not commute} with the~$\der^{k-j}$ in the
second summand of~\eqref{eq:der-fri}.

\begin{proof}
  We use induction on~$k$. In the base case~$k=0$, this is the usual
  variation-of-constants formula as given
  in~\cite[p.~74]{CoddingtonLevinson1955};
  see~\cite[Prop.~22]{RosenkranzRegensburger2008a}
  and~\cite[Thm.~6.4]{RegensburgerRosenkranz2009a} for its operator
  formulation. Now assume~\eqref{eq:der-fri} for fixed~$k \ge 0$; we
  show it for~$k+1$. By the induction hypothesis we obtain
  \begin{equation*}
    \der^{k+1} \fri{T} = \sum_{j=1}^n u_j^{(k+1)} \cum \frac{d_j}{d} + 
    \frac{1}{d} \sum_{j=1}^n u_j^{(k)} d_j + \sum_{j=1}^k
    \der^{k-j+1} \rho_j,
  \end{equation*}
  which is just~\eqref{eq:der-fri} for~$k+1$ since the middle
  sum is~$\rho_{k+1}$ and can be absorbed into the third.
\end{proof}

\begin{lemma}
  \label{lem:green-op}
  The Green's operator of any regular Stieltjes boundary problem is
  contained in~$\eqcumop + \mathcal{L}$, where~$\mathcal{L}$ denotes
  the left $\galg$-module generated by the local Stieltjes conditions.
\end{lemma}
\begin{proof}
  Assume~$(T, \bspc)$ is any regular Stieltjes boundary problem of
  order~$n$ with Green's operator~$G$, and let~$P$ be the projector
  onto~$\ker T$ along~$\orth{\bspc}$. By the proof
  of~\cite[Thm.~26]{RosenkranzRegensburger2008a} we
  have~$G=(1-P)\fri{T}$, and we know that~$P$ is an~$\galg$-linear
  combination of Stieltjes conditions
  by~\cite[(16)]{RosenkranzRegensburger2008a} in that same
  proof. From~\eqref{eq:der-fri} it is clear that~$\fri{T} \in
  \eqcumop$, so it suffices to show~$P \fri{T} \in \eqcumop +
  \mathcal{L}$. Each summand of~$P$ is either of the form~$f \lfloor
  \alpha \rfloor \, \der^k$ or~$f \lfloor \alpha \rfloor \cum g = f
  \cum_{\!\evl}\, g - f \cum_{\!\alpha}\, g \in \eqcumop$. In the
  latter case we obtain an expression in~$\eqcumop$ since~$\eqcumop$
  is a (nonunital) subalgebra of~$\eqintdiffop$. It remains to
  prove~$f \lfloor \alpha \rfloor \, \der^k \fri{T} \in \eqcumop +
  \mathcal{L}$. From~\eqref{eq:der-fri} we see that
  \begin{equation*}
    f \lfloor \alpha \rfloor \, \der^k \fri{T} = \sum_{j=1}^n f
    u_{j}^{(k)}(\alpha) \,
    \cum_{\!\evl}^\alpha \frac{d_j}{d} + \sum_{j=1}^k f \lfloor \alpha
    \rfloor \, \der^{k-j} \rho_j .
  \end{equation*}
  The first sum is clearly contained in~$\eqcumop$, while the second
  is in~$\mathcal{L}$ because~$\der^{k-j} \rho_j \in \diffop$ may be
  rewritten in canonical form as a sum of terms~$g_i \der^i$ so
  that~$\lfloor \alpha \rfloor \, \der^{k-j} \rho_j$ is a sum of local
  conditions~$\alpha(g_i) \, \lfloor \alpha \rfloor \, \der^i$ and hence
  itself local.
\end{proof}

We are now ready to state the main \emph{structure theorem for
Green's functions} of regular Stieltjes boundary problems.

\begin{theorem}
  \label{thm:greens-function}
  The Green's function of any regular Stieltjes boundary problem with
  $m$~evaluations~$\alpha_1, \dots, \alpha_m$ has the form~$g(x,\xi) =
  \tilde{g}(x,\xi) + \hat{g}(x,\xi)$, where the functional
  part~$\tilde{g} \in C(J^2)$ is defined by the $2(m-1)$~case
  branches
  \begin{align*}
    & \xi \in [\alpha_i, \alpha_{i+1}] \; (0<i<m), x\le\xi;\\
    & \xi \in [\alpha_i, \alpha_{i+1}] \; (0<i<m), \xi \le x,
  \end{align*}
  while the distributional part~$\hat{g}(x,\xi)$ is an~$\galg$-linear
  combination of the~$\delta(\xi-\alpha_i)$ and their derivatives.
\end{theorem}
\begin{proof}
  If~$G$ is the Green's operator of the given Stieltjes boundary
  problem, Lemma~\ref{lem:green-op} says that~$G = \tilde{G} +
  \hat{G}$ with~$\tilde{G} \in \eqcumop$ and~$\hat{G} \in
  \mathcal{L}$. We will show that~$\tilde{g}(x,\xi) =
  \tilde{G}_{x\xi}$ and~$\hat{g}(x,\xi) = \hat{G}_{x\xi}$ are as
  described in the theorem. Starting with the former, we may write
  \begin{equation*}
    \tilde{G} = \sum_{i=1}^r f_i \cum_{\!\alpha_i} g_i,
  \end{equation*}
  where $\alpha_i = \alpha_j$ is possible for~$i \neq j$. Using the
  transformation~$\eqintdiffop \to \mathcal{G}$, we
  obtain~$\tilde{g}(x,\xi)$ as
  \begin{align*}
    \sum_{i=1}^r \Big( & f_i(x) \, g_i(\xi) \, [\alpha_i \le
    \xi] [\xi \le x] - f_i(x) \, g_i(\xi) \, [\xi \le
    \alpha_i] [x \le \xi] \Big)\\
    & = \sum_{i=1}^r \Big( \sum_{\alpha_j \le \alpha_i} f_j(x) \,
    g_j(\xi) \Big) [\alpha_{j-1} \le \xi \le \alpha_j]
    [\xi \le x]\\
    & \quad - \sum_{i=1}^r \Big( \sum_{\alpha_j \ge \alpha_i}
    f_j(x) \, g_j(\xi) \Big) [\alpha_j \le \xi \le \alpha_{j+1}]
    [x \le \xi],
  \end{align*}
  where the two inner sums are restricted by~$j>0$
  and~$j<n$. Collecting terms, this is a sum
  of~$2(m-1)$~characteristic functions over disjoint domains
  in~$\RR^2$, hence one may also write~$\tilde{g}(x,\xi)$ in terms of
  a corresponding case distinction with $2(m-1)$~branches.

  The distributional part~$\hat{g}(x,\xi)$ is even
  easier. Writing~$\hat{G}$ as an~$\galg$-linear combination of local
  conditions we obain~$\hat{g}(x,\xi)$ via
  \begin{equation*}
    \hat{G}_{x\xi} = \Big( \sum_{\alpha,i} f_{i,\alpha} \alpha
    \der^i \Big)_{\!x\xi} = \sum_{\alpha,i} (-1)^i \, f_{i,\alpha}(x)
    \, \delta^{(i)}(\xi-\alpha),
  \end{equation*}
  which is clearly of the stated form.
\end{proof}

The above theorem is constructive, and we plan to implement the
underlying algorithm on top of the \texttt{Maple} package
\texttt{IntDiffOp}~\cite{KorporalRegensburgerRosenkranz2010}.

\begin{remark}
  \label{rem:add-cases}
  If the distinguished character~$\evl = \lfloor 0 \rfloor$ is not used in the boundary conditions, a straightforward translation of the Green's operator~$G$ may introduce two spurious extra case branches in the Green's function~$G_{x\xi}$ since~$\cum_{\!0}^x$ occurs in the formula
  for~$G$. For avoiding this, one has to use a different version of~$\fri{T}$ that replaces~$\lfloor 0 \rfloor$ by any one of the
  characters~$\lfloor \alpha_i \rfloor$ used in the boundary conditions.
\end{remark}

\section{Examples}

Our first example (in addition to the minimal one from the
Introduction) is a four-point boundary value problem taken
from~\cite[(2.1,2.2)]{BaiGeWang2005}, where we have specialized the
parameters and rescaled the interval to~$J = [0,1]$ for the sake of
simplicity. Hence we are dealing with the boundary problem
\begin{equation*}
  \label{eq:bvp-example1}
  \bvp{-u'' = f,}{u(0) + u(1/3) = u(1) + u(2/3) = 0,}
\end{equation*}
where we may assume~$u,f \in C^\infty[-2,2]$. Note that this is a
well-posed boundary problem, so the Green's function will not have a
distributional part. Computing the Green's operator with the
\texttt{IntDiffOp} package yields after some rearrangements the result
\begin{align*}
  &G = x\cum-\cum x+(-5/24+x/4) \lfloor 1/3 \rfloor\cum\\
  &+(5/8-3x/4) \lfloor 1/3 \rfloor \cum x +(1/8-3x/4) \lfloor 1 \rfloor \cum x\\
  & + (1/12-x/2)\lfloor 2/3 \rfloor \cum +(-1/8+3x/4)\lfloor 2/3 \rfloor \cum x
\end{align*}
 Transforming~$G$ to equitable form is simple, via~$\lfloor
\alpha \rfloor \cum \rightsquigarrow \cum_0^x - \cum_\alpha^x$. We can
then determine the corresponding Green's function~$g(x,\xi) =
\tilde{g}(x,\xi)$ with~$6$ cases, and its terms may be computed according
to Theorem~\ref{thm:greens-function}. The result for~$g(x,\xi)$ is summarized in the table below.

{
  \scriptsize
  \begin{equation*}
    \begin{array}{|l|l|}
      \hline
      \text{Case} & \text{Term}\\\hline
      0 \le \xi \le 1/3, \xi \le x & (3/4)x\xi-(5/8)\xi \\\hline
      0 \le \xi \le 1/3, x \le \xi &(3/4)x\xi+(3/8)\xi-x\\\hline
      1/3 \le \xi \le 2/3, \xi \le x & (3/2)x\xi-(5/4)\xi-(1/4)x+5/24 \\\hline
      1/3 \le \xi \le 2/3, x \le \xi & (3/2)x\xi-(1/4)\xi-(5/4)x+ 5/24\\\hline
      2/3 \le \xi \le 1, \xi \le x &(3/4)x\xi-(9/8)\xi+(1/4)x+1/8\\\hline
      2/3 \le \xi \le 1, x \le \xi & (3/4)x\xi-(1/8)\xi-(3/4)x+1/8\\\hline
    \end{array}
  \end{equation*}
}

Our second example is, as it were, totally unclassical: It is
ill-posed, has nonlocal conditions and contains three evaluation
points~$-1,0,1$. In our standard notation, we write this boundary
problem as
\begin{equation*}
  \label{eq:bvp-example2}
  \bvp{u''-u=f,\vspace{0.5ex}}{u'''(-1)-\int_0^1 u(\xi) \, \xi \, d\xi = 0,\\
    u'(-1) - u''(1) + \int_{-1}^1 u(\xi) \, d\xi = 0,}
\end{equation*}
where we assume now~$u,f \in C^\infty[-1,1]$. Using the method
of~\cite{RosenkranzRegensburger2008a}, it is straightforward to
compute the Green's operator~$G$. In fact, the \texttt{IntDiffOp}
package yields the result
\begin{align*}
  \label{eq:bvp-sol2}
  & \sigma \, G = \sigma /2 (e^x \cum e^{-x} - e^{-x} \cum e^x)\\
  & + 2 (-e^{x+3} + e^{x+2} - e^{x+1} + e^{-x+2} - e^{-x+1}) (\lfloor -1
  \rfloor \der + \lfloor 1 \rfloor \cum x)\\
  & + (e-1) (-e^{x+2} - 2 e^{x+1} + e^{-x+1})(\lfloor -1 \rfloor \cum +
  \lfloor 1 \rfloor \cum)\\
  & + (3e^{x+2} - e^{x+1} - 3e^{-x+1} + 3e^{-x}) \lfloor 1 \rfloor \cum e^x\\
  & + (2e^{x+2}-3e^{x+1})(e^{-1}\lfloor -1 \rfloor \cum e^{-x}+ e\lfloor -1 \rfloor \cum e^{x})\\
  &+ (-e^{x+3}-e^{x+2}+2e^{x+1}+e^{-x+2}-e^{-x+1})\lfloor 1 \rfloor
\end{align*}
using the abbreviation~$\sigma := 2(2e-3)(e-1)$ while collecting
and factoring some terms for enhanced readability. After transforming
this to equitable form (which is again straightforward), we can apply
Theorem~\ref{thm:greens-function} to extract the Green's
function~$g(x,\xi) = \tilde{g}(x,\xi) + \hat{g}(x,\xi)$ with the
distributional part
\begin{align*}
  &\sigma \, \hat{g}(x,\xi) =  (-e^{x+3}-e^{x+2}+ 2 e^{x+1}+e^{-x+2}-e^{-x+1})
  \, \delta(\xi-1)\\
  & +2 \, (-e^{x+3} + e^{x+2} - e^{x+1} + e^{-x+2} - e^{-x+1}) \,
  \delta'(\xi-1)
\end{align*}
coming from the~$(\dots) \, \lfloor 1 \rfloor$ and~$(\dots) \, \lfloor 1
\rfloor \der$ terms, and with the functional part defined by the case
distinction for~$\sigma \, \tilde{g}(x,\xi)$ as given in
the table below.
{
  \scriptsize
  \begin{equation*}
    \begin{array}{|l|l|}
      \hline
      \text{Case} & \text{Term}\\\hline
      -1 \le \xi \le 0, \xi \le x & 3e^{x+2+\xi}+3e^{x-\xi}-2e^{x+1-\xi}-2e^{3+x+\xi}\\
      & \quad+e^{3+x}+e^{-x+1}+e^{x+2}-e^{-x+2}-2e^{x+1}\\\hline
      -1 \le \xi \le 0, x \le \xi &-2e^{x+1}+2e^{-x+2+\xi}-5e^{-x+1+\xi}-2e^{x+2-\xi}\\
      & \quad-2e^{3+x+\xi}+3e^{-x+\xi}+e^{-x+1}+e^{x+2}\\
      & \quad+e^{3+x}+3e^{x+1-\xi}+3e^{x+2+\xi}-e^{-x+2}\\\hline
      0 \le \xi \le 1, \xi \le x &-2e^{3+x}\xi-2e^{-x+1}\xi+2e^{x+2}\xi+2e^{-x+2}\xi\\
      & \quad -2e^{x+1}\xi+3e^{x+2+\xi}+3e^{x-\xi}-5e^{x+1-\xi}\\
         & \quad +2e^{-x+1+\xi}-e^{x+1+\xi}-2e^{-x+2+\xi}+2e^{x+2-\xi}\\
      & \quad-e^{3+x}-e^{-x+1}-e^{x+2}+e^{-x+2}+2e^{x+1}\\\hline
      0 \le \xi \le 1, x \le \xi &-2e^{3+x}\xi-2e^{-x+1}\xi+2e^{x+2}\xi+2e^{-x+2}\xi\\
      & \quad-2e^{x+1}\xi+3e^{-x+\xi}+3e^{x+2+\xi}-e^{3+x}\\
      & \quad-e^{-x+1}-e^{x+2}+e^{-x+2}+2e^{x+1}\\
      &\quad -3e^{-x+1+\xi}-e^{x+1+\xi} \\\hline
    \end{array}
  \end{equation*}
}

Incidentally, this example shows also that the \emph{representation of
  Green's operators} in terms of Green's functions---despite its long
tradition in engineering and physics---is not always the most useful
and economical way of representing the Green's operator. For many
purposes it is better to take the Green's operator just as an element
of the operator ring~$\intdiffop$ or~$\eqintdiffop$.

\section{Conclusion and Future Work}
\label{sec:conclusion}

While we have focused our attention to semi-inhomogeneous boundary
problems (those with an inhomogeneous differential equation and
homogeneous boundary conditions), one may also consider the opposite
case of semi-homogeneous boundary problems---this is especially
important in the case of LPDEs. The corresponding semi-homogeneous
Green's operator maps the prescribed boundary values to the
solution~\cite{RosenkranzPhisanbut2013}. In the case of LODEs, one
often restricts attention to well-posed two-point boundary value
problems (in the sense of Definition~\ref{def:well-posed}). Writing
the two evaluations as~$\lfloor a \rfloor$ and~$\lfloor b \rfloor$ and
their action on~$u$ as~$u(a)$ and~$u(b)$, one may consider the
extended evaluation matrix
\begin{equation*}
  \lfloor a, b \rfloor(u) := 
  \begin{pmatrix}
    u_1(a) & \cdots & u_n(a)\\
    \vdots & \ddots & \vdots\\
    u_1^{(n-1)}(a) & \cdots & u_n^{(n-1)}(a)\\
    u_1(b) & \cdots & u_n(b)\\
    \vdots & \ddots & \vdots\\
    u_1^{(n-1)}(b) & \cdots & u_n^{(n-1)}(b)
  \end{pmatrix}
  \in K^{2n \times n}
\end{equation*}
which is similar to~\eqref{eq:ev-matrix} except that it is rectangular
since we consider more boundary functionals than we could possibly
impose for one regular boundary problem. If we do prescribe all~$2n$
boundary derivatives, they must satisfy~$n$ relations given by the
kernel of the map~$\RR^{2n} \to \RR^n, X \mapsto X \cdot \lfloor a, b
\rfloor (u)$. For the simple example in Section~\ref{sec:intro}, the
extended evaluation matrix for the fundamental system~$u_1 = 1, u_2 =
x$ is
\begin{equation*}
  \lfloor 0,1 \rfloor(u) = 
  \begin{pmatrix}
    1 & 0\\
    0 & 1\\
    1 & 1\\
    0 & 1
  \end{pmatrix}
\end{equation*}
whose kernel has basis~$(-1,1,1,0), (0,-1,0,1)$. Written in terms of
the boundary functionals, they encode the two relations~$u(1) - u(0) =
u'(0)$ and~$u'(0) = u'(1)$. The analogous case for LPDEs gives rise to
the interesting notion of universal boundary
problem~\cite{VolovichSakbaev2014}.

There is another, more fundamental, way of extending the results in
this paper: Currently our method for extracting Green's functions
works for arbitrary Stieltjes boundary problems~$(T, \bspc)$, but only
in the \emph{standard integro-differential algebra}~$\galg =
C^\infty(\RR)$. Of course, it requires a Green's operator~$G \in
\intdiffop \cong \eqintdiffop$ and hence a fundamental system for~$T$.

It would be interesting to extend the concept of Green's function and
the corresponding extraction method to arbitrary ordinary
integro-differential algebras~$(\galg, \der, \cum)$. For the
functional part~$\tilde{g}(x,\xi)$, it is clear how to achieve this
since one sees from the structure of Green's operators that
necessarily~$\tilde{g} \in \galg \otimes \galg$. The ring~$\galg
\otimes \galg$ has the structure of a partial integro-differential
algebra with derivations and integrals
\begin{align*}
  & \der_x (f \otimes g) = (\der f) \otimes g
  \quad\text{and}\quad
  \der_y (f \otimes g) = f \otimes (\der g),\\
  & \cum^x (f \otimes g) = (\cum f) \otimes g
  \quad\text{and}\quad
  \cum^y (f \otimes g) = f \otimes (\cum g).
\end{align*}
This structure will be useful for studying various properties of
Green's function, in particular their symmetry: For well-posed
two-point boundary value problems~$(T,\bspc)$ it is
known~\cite[\S3.3]{Stakgold2011} that the Green's function~$g(x,\xi)$
is symmetric whenever~$(T, \bspc)$ is \emph{self-adjoint}. Otherwise
one may associate to~$(T, \bspc)$ an \emph{adjoint boundary value
  problem} whose Green's function is then~$g(\xi,x)$. It would be
useful to know how these results generalizes to arbitrary Stieltjes
boundary problems.

Having an abstract integro-differential algebras~$(\galg, \der,
\cum)$, the other problem is Green's function will in general have a
distributional part~$\hat{g}(x,\xi)$ that does not fit into~$\galg
\otimes \galg$. For accommodating distributions into the setting of
integro-differential algebras, it is probably necessary to construct a
integro-differential module generated over~$\galg \otimes \galg$ by a
suitable notion of \emph{abstract Dirac distributions}. (It is well
known that distributions do not enjoy a convenient ring structure,
hence it seems to be more reasonable to go for a module. This is also
the path followed in the algebraic analysis of $\mathcal{D}$-modules;
see~\cite[\S6.1]{Coutinho1995} for example.)

\section*{Acknowledgments}

We would like to thank the referees for their diligent work and their
crucial suggestions. In particular, the reference on the universal
boundary value problem (see first paragraph of
Section~\ref{sec:conclusion}) is of considerable interest.

\bibliographystyle{abbrv}
\bibliography{BPGF}


\end{document}